\documentclass[12pt]{amsart}
\usepackage{amsmath}
\usepackage{amsthm}
\usepackage{amssymb}
\usepackage{amsfonts}
\usepackage{bm}
\usepackage[shortlabels]{enumitem}
\usepackage[bookmarks=true,hyperindex,pdftex,colorlinks,citecolor=red,linkcolor=blue]{hyperref}
\usepackage{centernot}
\usepackage{marginnote}
\usepackage{bbm}
\usepackage[all]{xy}
\usepackage{tikz}
\usepackage{comment}
\usetikzlibrary{arrows}
\usetikzlibrary{shapes,decorations}
\usetikzlibrary{positioning}
\usepackage{xcolor}
\usepackage{bigints}
\usepackage{cancel}
\usepackage[
a4paper,
left=2.4cm,
right=2.4cm,
top=2.3cm,
bottom=2.5cm,
headheight=14pt,
includeheadfoot
]{geometry}
\usepackage{doi}

\theoremstyle{plain}
\newtheorem{thm}{Theorem}

\newtheorem{lemma}{Lemma}

\theoremstyle{definition}

\theoremstyle{remark}

\newcommand{\C}{\mathbb{C}}

\newcommand{\R}{\mathbb{R}}
\newcommand{\Z}{\mathbb{Z}}
\newcommand{\T}{\mathbb{T}}

\newcommand{\eps}{\varepsilon}
\newcommand{\norm}[1]{\left\Vert #1\right\Vert}

\definecolor{darkgreen}{rgb}{.2,.6,.2}
\title{Uniform Kreiss boundedness does not imply strong Ces\`aro boundedness}

\author[L. Arnold]{Loris Arnold}

\address[L. Arnold]{Normandie Univ, UNICAEN, CNRS, LMNO, 14000 Caen, France}
\email{lfj.arld@gmail.com}

\subjclass[2020]{47A35, 47A10}
\keywords{Uniform Kreiss boundedness, strong Ces\`aro boundedness}

\begin{document}
	
	\begin{abstract}
		We construct a uniformly Kreiss bounded operator on a separable Hilbert space
		which is not strongly Ces\`aro bounded. This gives a negative answer to a
		question of Cohen, Cuny, Eisner and Lin.
	\end{abstract}
	
	\maketitle
	
	\section{Introduction}
	
	The notion of uniform Kreiss boundedness is a natural intermediate condition between power boundedness and the ordinary Kreiss resolvent condition.
	Let $X$ be a complex Banach space, let $B(X)$ denote the algebra of bounded
	linear operators on $X$ and let $T\in B(X)$. The operator $T$ is
	\emph{uniformly Kreiss bounded} if there exists $C>0$ such that
	\[
	\left\|
	\sum_{k=0}^{n}\lambda^{-k-1}T^k
	\right\|
	\leq \frac{C}{|\lambda|-1},
	\qquad n\geq0,\quad |\lambda|>1.
	\]
	Montes-Rodr\'{\i}guez, S\'anchez-\'Alvarez and Zem\'anek
	\cite[Corollary~3.2]{MontesSanchezZemanek} showed that uniform Kreiss
	boundedness can equivalently be formulated in terms of rotated Ces\`aro
	means:
	\begin{equation}\label{eq:UKB-characterization}
		\sup_{m\geq1}
		\sup_{\zeta\in\T}
		\left\|
		\frac1m\sum_{k=1}^{m}(\zeta T)^k
		\right\|
		<\infty.
	\end{equation}

	More recently, Cohen, Cuny, Eisner and Lin
	\cite{CohenCunyEisnerLin} introduced \emph{strong Ces\`aro boundedness}:
	$T$ is strongly Ces\`aro bounded if there exists $C>0$ such that
	\[
	\frac1N\sum_{k=0}^{N-1}
	\bigl|\langle x^*,T^kx\rangle\bigr|
	\leq C\norm{x}\norm{x^*},
	\qquad x\in X,\quad x^*\in X^*,\quad N\geq1.
	\]
	By \cite[Proposition~3.6]{CohenCunyEisnerLin}, this is equivalent to
	\begin{equation}\label{eq:SCB-characterization}
		\sup_{N\geq1}
		\sup_{\gamma_0,\ldots,\gamma_{N-1}\in\T}
		\left\|
		\frac1N\sum_{k=0}^{N-1}\gamma_kT^k
		\right\|
		<\infty.
	\end{equation}
	In particular, strong Ces\`aro boundedness implies uniform Kreiss boundedness. 
	
	Uniform Kreiss boundedness has been studied in connection with power growth,
	ergodic properties and Ces\`aro boundedness, see, among others,
	\cite{AlemanSuciu,BermudezBonillaMullerPeris,BonillaMuller}.
	In particular, Bonilla and M\"uller showed that, for every $\eps>0$, there
	exists a uniformly Kreiss bounded operator on a Hilbert space such that
	\[
	\norm{T^n}\asymp (n+1)^{1-\eps}.
	\]
	Interestingly, these examples are nevertheless strongly Ces\`aro bounded,
	and therefore do not separate uniform Kreiss boundedness from strong
	Ces\`aro boundedness.
	
	Since strong Ces\`aro boundedness implies uniform Kreiss boundedness, Cohen,
	Cuny, Eisner and Lin asked whether the converse can fail, preferably on a
	Hilbert space, see \cite[Section~6, Problem~3]{CohenCunyEisnerLin}. We show
	that it does, even on a separable Hilbert space.
	
	\begin{thm}\label{thm:main}
		There exists a uniformly Kreiss bounded operator on a separable Hilbert space
		that is not strongly Ces\`aro bounded.
	\end{thm}
	
	\section{The counterexample}
	
	The construction relies on finite-dimensional bases recently obtained by
	Lorist, Meyries and Veraar \cite[Proposition~2.1]{LoristMeyriesVeraar}.
	For every $0<\alpha<1$ and every $n\geq1$, there exists a basis
	$(z_{n,j})_{j=1}^{2n}$ of $H_n:=\C^{2n}$ such that the partial-sum
	projections
	\[
	P_{n,r}\left(\sum_{j=1}^{2n}c_jz_{n,j}\right)
	:=
	\sum_{j=1}^{r}c_jz_{n,j},
	\qquad 1\leq r\leq2n,
	\]
	satisfy
	\begin{equation}\label{eq:partial-sums}
		\sup_{n\geq1}\sup_{1\leq r\leq2n}\norm{P_{n,r}}
		\leq K_\alpha,
		\qquad
		K_\alpha:=\frac{5}{1-\alpha}.
	\end{equation}
	For $\mathbf{a}=(a_j)_{j=1}^{2n}\in\C^{2n}$, define
	\[
	\Delta_{n,\mathbf{a}}z_{n,j}:=a_jz_{n,j},
	\qquad 1\leq j\leq2n.
	\]
	Then
	\begin{equation}\label{eq:general-multiplier}
		\norm{\Delta_{n,\mathbf{a}}}
		\leq
		4n^\alpha\max_{1\leq j\leq2n}|a_j|.
	\end{equation}
	Moreover, the alternating-sign multiplier
	\[
	\Delta_{n,\pm}z_{n,j}:=(-1)^jz_{n,j}
	\]
	satisfies
	\begin{equation}\label{eq:alternating-multiplier}
		\norm{\Delta_{n,\pm}}
		\geq\frac12n^\alpha.
	\end{equation}
	Finally, summation by parts gives
	\begin{equation}\label{eq:summation-by-parts}
		\Delta_{n,\mathbf{a}}
		=
		a_{2n}I+
		\sum_{j=1}^{2n-1}(a_j-a_{j+1})P_{n,j}.
	\end{equation}
	
	We shall also use two elementary lemmas. Their proofs are postponed to the
	appendix in order to keep the construction transparent. The first is the
	analytic ingredient responsible for uniform Kreiss boundedness.
	
	\begin{lemma}\label{lem:lacunary-variation}
		Let $0<q<1$. There exists $V_q>0$ such that, whenever $m,N\geq1$ and
		\[
		0<\theta_N<\cdots<\theta_2<\theta_1\leq\frac{\pi}{4},
		\qquad
		\theta_{j+1}\leq q\theta_j
		\quad(1\leq j<N),
		\]
		one has
		\begin{equation}\label{eq:lacunary-variation}
			\sup_{\varphi\in\R}
			\sum_{j=1}^{N-1}
			\left|
			F_m(\varphi+\theta_{j+1})
			-F_m(\varphi+\theta_j)
			\right|
			\leq V_q,
		\end{equation}
		where
		\[
		F_m(x):=\frac1m\sum_{k=1}^{m}e^{ikx}.
		\]
	\end{lemma}
	
	The second lemma provides unimodular coefficients for which the corresponding Ces\`aro averages approximate a scalar multiple of the alternating-sign
	multiplier.
	
	\begin{lemma}\label{lem:phase-selection}
		Let $d\geq1$ and let $\eps_1,\ldots,\eps_d\in\{-1,1\}$. There exist a
		measurable function $\Gamma:\T^d\to\T$ and a number $c_d>0$ such that
		\begin{equation}\label{eq:phase-Fourier}
			\int_{\T^d}\Gamma(w)w_j\,dm_d(w)
			=
			c_d\eps_j,
			\qquad 1\leq j\leq d,
		\end{equation}
		where $m_d$ denotes normalized Haar measure on $\T^d$. Moreover,
		\begin{equation}\label{eq:cd-lower}
			c_d\geq\frac1{\sqrt{2d}},
		\end{equation}
		and each function $w\mapsto\Gamma(w)w_j$ is Riemann integrable on $\T^d$.
	\end{lemma}
	
	\begin{proof}[Proof of Theorem~\ref{thm:main}]
		Fix
		\[
		\frac12<\alpha<1
		\]
		and choose once and for all $q\in(0,1)$, for instance $q=1/4$. For every
		$n\geq1$, let $(z_{n,j})_{j=1}^{2n}$ be a basis of $H_n=\C^{2n}$ satisfying
		\eqref{eq:partial-sums}--\eqref{eq:alternating-multiplier}.
		
		Choose numbers
		\[
		0<\theta_{n,2n}<\cdots<\theta_{n,1}<\frac{\pi}{4}
		\]
		such that
		\begin{equation}\label{eq:theta-lacunarity}
			\theta_{n,j+1}\leq q\theta_{n,j},
			\qquad 1\leq j<2n,
		\end{equation}
		and
		\begin{equation}\label{eq:rational-independence}
			1,\frac{\theta_{n,1}}{2\pi},\ldots,
			\frac{\theta_{n,2n}}{2\pi}
			\quad\text{are linearly independent over }\mathbb Q.
		\end{equation}
		Such a family can be chosen recursively. Indeed, after
		$\theta_{n,1},\ldots,\theta_{n,j}$ have been selected, it suffices to choose
		$\theta_{n,j+1}\in(0,q\theta_{n,j})$ outside
		\[
		\operatorname{span}_{\mathbb Q}
		\{2\pi,\theta_{n,1},\ldots,\theta_{n,j}\},
		\]
		which is countable.
		
		Define $T_n\in B(H_n)$ by
		\[
		T_nz_{n,j}=e^{i\theta_{n,j}}z_{n,j},
		\qquad 1\leq j\leq2n.
		\]
		Using \eqref{eq:summation-by-parts}, \eqref{eq:partial-sums} and
		$|e^{is}-e^{it}|\leq|s-t|$, we obtain
		\[
		\norm{T_n}
		\leq
		1+K_\alpha
		\sum_{j=1}^{2n-1}
		|e^{i\theta_{n,j}}-e^{i\theta_{n,j+1}}|
		\leq
		1+K_\alpha\theta_{n,1}
		\leq
		1+\frac{\pi}{4}K_\alpha.
		\]
		Hence
		\[
		H:=\left(\bigoplus_{n\geq1}H_n\right)_{\ell^2},
		\qquad
		T:=\bigoplus_{n\geq1}T_n,
		\]
		defines a bounded operator on a separable Hilbert space.
		
		We first prove that $T$ is uniformly Kreiss bounded. Fix $m,n\geq1$ and
		$\zeta=e^{i\varphi}\in\T$. On the basis $(z_{n,j})_{j=1}^{2n}$,
		\[
		\frac1m\sum_{k=1}^{m}(\zeta T_n)^kz_{n,j}
		=
		F_m(\varphi+\theta_{n,j})z_{n,j}.
		\]
		Thus this Ces\`aro mean is the multiplier $\Delta_{n,\mathbf{a}}$ with
		\[
		a_j=F_m(\varphi+\theta_{n,j}).
		\]
		By \eqref{eq:summation-by-parts}, \eqref{eq:partial-sums} and
		Lemma~\ref{lem:lacunary-variation},
		\[
		\left\|
		\frac1m\sum_{k=1}^{m}(\zeta T_n)^k
		\right\|
		\leq
		1+K_\alpha
		\sum_{j=1}^{2n-1}
		|F_m(\varphi+\theta_{n,j})
		-F_m(\varphi+\theta_{n,j+1})|
		\leq
		1+K_\alpha V_q.
		\]
		Since the direct sum is orthogonal, the same estimate holds for $T$.
		Taking the supremum over $m\geq1$ and $\zeta\in\T$ and using
		\eqref{eq:UKB-characterization}, we conclude that $T$ is uniformly Kreiss
		bounded.
		
		It remains to prove that $T$ is not strongly Ces\`aro bounded. Fix $n\geq1$,
		put $d:=2n$ and let $\eps_j:=(-1)^j$. Apply
		Lemma~\ref{lem:phase-selection} to obtain $\Gamma:\T^d\to\T$ and $c_d$
		satisfying \eqref{eq:phase-Fourier} and \eqref{eq:cd-lower}. By a classical
		result of Weyl, see for instance \cite[p.~209]{Colzani}, \eqref{eq:rational-independence} implies that
		\[
		k\longmapsto
		\bigl(e^{ik\theta_{n,1}},\ldots,e^{ik\theta_{n,d}}\bigr),
		\qquad k\geq0,
		\]
		is equidistributed in $\T^d$. Since $w\mapsto\Gamma(w)w_j$ is Riemann
		integrable, it follows that
		\[
		\frac1M\sum_{k=0}^{M-1}
		\Gamma\bigl(e^{ik\theta_{n,1}},\ldots,e^{ik\theta_{n,d}}\bigr)
		e^{ik\theta_{n,j}}
		\longrightarrow
		\int_{\T^d}\Gamma(w)w_j\,dm_d(w)
		=
		c_d(-1)^j
		\]
		for every $1\leq j\leq 2n$.
		
		Set
		\[
		\gamma_k^{(n)}
		:=
		\Gamma\bigl(e^{ik\theta_{n,1}},\ldots,e^{ik\theta_{n,d}}\bigr),
		\qquad k\geq0.
		\]
		Then $|\gamma_k^{(n)}|=1$. Since there are only finitely many indices
		$1\leq j\leq2n$, we may choose $M_n\geq1$ such that
		\begin{equation}\label{eq:sign-approximation}
			\max_{1\leq j\leq2n}
			\left|
			\frac1{M_n}\sum_{k=0}^{M_n-1}
			\gamma_k^{(n)}e^{ik\theta_{n,j}}
			-c_{2n}(-1)^j
			\right|
			\leq\frac{c_{2n}}{16}.
		\end{equation}
		
		Let
		\[
		A_n
		:=
		\frac1{M_n}\sum_{k=0}^{M_n-1}\gamma_k^{(n)}T_n^k.
		\]
		For $1\leq j\leq2n$, set
		\[
		r_{n,j}
		:=
		\frac1{M_n}\sum_{k=0}^{M_n-1}
		\gamma_k^{(n)}e^{ik\theta_{n,j}}
		-c_{2n}(-1)^j,
		\]
		and
		\[
		R_n:=\Delta_{n,\mathbf{r}_n},
		\qquad
		\mathbf{r}_n=(r_{n,j})_{j=1}^{2n}.
		\]
		Then
		\[
		A_n=c_{2n}\Delta_{n,\pm}+R_n,
		\]
		and \eqref{eq:sign-approximation} gives
		\[
		\max_{1\leq j\leq2n}|r_{n,j}|
		\leq\frac{c_{2n}}{16}.
		\]
		Hence, by \eqref{eq:general-multiplier},
		\[
		\norm{R_n}
		\leq
		\frac14c_{2n}n^\alpha.
		\]
		Using \eqref{eq:alternating-multiplier}, we therefore obtain
		\[
		\norm{A_n}
		\geq
		c_{2n}\norm{\Delta_{n,\pm}}-\norm{R_n}
		\geq
		\frac14c_{2n}n^\alpha.
		\]
		Since \eqref{eq:cd-lower} with $d=2n$ gives
		\[
		c_{2n}\geq\frac1{2\sqrt n},
		\]
		we obtain
		\begin{equation}\label{eq:large-Cesaro}
			\norm{A_n}
			\geq
			\frac18n^{\alpha-\frac12}
			\longrightarrow\infty.
		\end{equation}
		
		Finally, define
		\[
		\widetilde A_n
		:=
		\frac1{M_n}\sum_{k=0}^{M_n-1}\gamma_k^{(n)}T^k
		\in B(H).
		\]
		Since $H_n$ is invariant under $T$ and $T^k|_{H_n}=T_n^k$, we have
		\[
		\widetilde A_n|_{H_n}=A_n.
		\]
		Consequently,
		\[
		\norm{\widetilde A_n}
		\geq
		\norm{A_n}
		\longrightarrow\infty.
		\]
		This contradicts \eqref{eq:SCB-characterization}, applied with
		$N=M_n$ and the unimodular sequence $(\gamma_k^{(n)})_{k\geq0}$.
		Therefore $T$ is not strongly Ces\`aro bounded.
	\end{proof}
	
	\appendix
	
	\section{Proof of Lemma~\ref{lem:lacunary-variation}}
	
	\begin{proof}
		There is nothing to prove when $N=1$, so assume $N\geq2$. Set
		\[
		D_m(x):=\frac1m\sum_{k=0}^{m-1}e^{ikx},
		\qquad
		F_m(x)=e^{ix}D_m(x),
		\]
		and write
		\[
		\delta(x):=\operatorname{dist}(x,2\pi\Z).
		\]
		By periodicity, it suffices to consider $\varphi\in[-\pi,\pi]$.
		
		For $0<|x|\leq\pi$,
		\[
		D_m(x)
		=
		\frac{1-e^{imx}}{m(1-e^{ix})},
		\qquad
		|1-e^{ix}|
		\geq
		\frac{2}{\pi}|x|.
		\]
		Hence, by periodicity,
		\begin{equation}\label{eq:Dm-size}
			|D_m(x)|
			\leq
			\min\left\{
			1,\frac{\pi}{m\delta(x)}
			\right\},
			\qquad
			\delta(x)>0.
		\end{equation}
		Moreover,
		\[
		D_m'(x)
		=
		\frac{i}{m}\sum_{k=0}^{m-1}ke^{ikx}
		\]
		gives
		\begin{equation}\label{eq:Dm-derivative-trivial}
			|D_m'(x)|\leq\frac m2.
		\end{equation}
		Differentiating the quotient formula also yields, for $0<|x|\leq\pi$,
		\[
		|D_m'(x)|
		\leq
		\frac{\pi}{2|x|}
		+
		\frac{\pi^2}{2m|x|^2}.
		\]
		Thus, for an absolute constant $C>0$,
		\begin{equation}\label{eq:Dm-derivative}
			|D_m'(x)|
			\leq
			\frac{C}{\delta(x)},
			\qquad
			m\delta(x)\geq1.
		\end{equation}
		
		We shall repeatedly use
		\begin{equation}\label{eq:reciprocal-lacunarity}
			\sum_{j=r}^{s}\frac1{\theta_j}
			\leq
			\frac1{1-q}\frac1{\theta_s},
			\qquad
			1\leq r\leq s\leq N.
		\end{equation}
		
		We first record two simple variation estimates. Suppose that $r\leq s$ and
		\begin{equation}\label{eq:distance-assumption}
			\delta(\varphi+\theta_j)
			\geq
			a\theta_j,
			\qquad r\leq j\leq s,
		\end{equation}
		for some $a>0$. Then
		\begin{equation}\label{eq:separated-variation}
			\sum_{j=r}^{s-1}
			|D_m(\varphi+\theta_{j+1})
			-D_m(\varphi+\theta_j)|
			\leq
			C\left(1+\frac1{a(1-q)}\right).
		\end{equation}
		Indeed, if all $\theta_j$ in the block are at least $m^{-1}$, then
		\eqref{eq:Dm-size}, \eqref{eq:distance-assumption} and
		\eqref{eq:reciprocal-lacunarity} give
		\[
		\sum_{j=r}^{s-1}
		|D_m(\varphi+\theta_{j+1})
		-D_m(\varphi+\theta_j)|
		\leq
		\frac{2\pi}{am}
		\sum_{j=r}^{s}\frac1{\theta_j}
		\leq
		\frac{2\pi}{a(1-q)}.
		\]
		If all $\theta_j<m^{-1}$, then \eqref{eq:Dm-derivative-trivial} gives
		\[
		\sum_{j=r}^{s-1}
		|D_m(\varphi+\theta_{j+1})
		-D_m(\varphi+\theta_j)|
		\leq
		\frac m2(\theta_r-\theta_s)
		\leq
		\frac12.
		\]
		In the remaining case, split at the unique transition from
		$\theta_j\geq m^{-1}$ to $\theta_j<m^{-1}$. The two parts are estimated as
		above, while the transition contributes at most $2$. This proves
		\eqref{eq:separated-variation}.
		
		Next, let $0<c\leq\pi$ and suppose
		\[
		0<\theta_s<\cdots<\theta_r\leq qc.
		\]
		Then
		\begin{equation}\label{eq:tail-variation}
			\sum_{j=r}^{s-1}
			|D_m(\theta_{j+1}-c)-D_m(\theta_j-c)|
			\leq C_q,
		\end{equation}
		where $C_q$ depends only on $q$. Indeed, the discrete variation is bounded
		by the total variation of $t\mapsto D_m(t-c)$ on $[0,qc]$. If
		$m(1-q)c<1$, then \eqref{eq:Dm-derivative-trivial} gives
		\[
		\int_0^{qc}|D_m'(t-c)|\,dt
		\leq
		\frac m2qc
		\leq
		\frac{q}{2(1-q)}.
		\]
		If $m(1-q)c\geq1$, then \eqref{eq:Dm-derivative} applies throughout
		$[0,qc]$ and yields
		\[
		\int_0^{qc}|D_m'(t-c)|\,dt
		\leq
		C\int_0^{qc}\frac{dt}{c-t}
		=
		C\log\frac1{1-q}.
		\]
		This proves \eqref{eq:tail-variation}.
		
		Set $L:=\theta_1$. We now estimate the variation of $D_m$ along the entire
		sequence.
		
		Suppose first that $\varphi\geq0$. Since
		$\varphi\in[0,\pi]$ and $0<\theta_j\leq\pi/4$,
		\[
		\delta(\varphi+\theta_j)\geq\theta_j,
		\qquad 1\leq j\leq N.
		\]
		Hence \eqref{eq:separated-variation} applies with $a=1$.
		
		Assume next that $\varphi\leq-L$ and put
		\[
		c:=-\varphi\in[L,\pi].
		\]
		Then
		\[
		\theta_2\leq q\theta_1\leq qc.
		\]
		The first increment is bounded by $2$, while
		\eqref{eq:tail-variation}, applied to $\theta_2,\ldots,\theta_N$, controls
		all remaining increments.
		
		It remains to consider
		\[
		-L<\varphi<0.
		\]
		Put $c:=-\varphi\in(0,L)$ and let
		\[
		r:=\max\{j\in\{1,\ldots,N\}:\theta_j>c\}.
		\]
		Since $c<L=\theta_1$, one has $1\leq r\leq N$. For every $j\leq r-1$,
		\[
		c<\theta_{j+1}\leq q\theta_j,
		\]
		and therefore
		\[
		\delta(\varphi+\theta_j)
		=
		\theta_j-c
		\geq
		(1-q)\theta_j.
		\]
		Thus \eqref{eq:separated-variation}, with $a=1-q$, controls the increments
		with indices $1\leq j\leq r-2$.
		
		The increments whose indices belong to
		\[
		\{r-1,r,r+1\}\cap\{1,\ldots,N-1\}
		\]
		are bounded trivially by $2$ each. Finally, whenever $r+2\leq N$,
		maximality of $r$ gives $\theta_{r+1}\leq c$ and hence
		\[
		\theta_{r+2}
		\leq
		q\theta_{r+1}
		\leq
		qc.
		\]
		Therefore \eqref{eq:tail-variation}, applied to
		$\theta_{r+2},\ldots,\theta_N$, controls the remaining tail. We have thus
		proved
		\begin{equation}\label{eq:Dm-variation}
			\sup_{\varphi\in\R}
			\sum_{j=1}^{N-1}
			|D_m(\varphi+\theta_{j+1})
			-D_m(\varphi+\theta_j)|
			\leq
			C_q.
		\end{equation}
		
		Since $F_m(x)=e^{ix}D_m(x)$ and $|D_m|\leq1$,
		\begin{align*}
			|F_m(\varphi+\theta_{j+1})
			-F_m(\varphi+\theta_j)|
			&\leq
			|D_m(\varphi+\theta_{j+1})
			-D_m(\varphi+\theta_j)| \\
			&\quad+
			|e^{i\theta_{j+1}}-e^{i\theta_j}|.
		\end{align*}
		Moreover,
		\[
		|e^{i\theta_{j+1}}-e^{i\theta_j}|
		\leq
		\theta_j-\theta_{j+1},
		\]
		and therefore
		\[
		\sum_{j=1}^{N-1}
		|e^{i\theta_{j+1}}-e^{i\theta_j}|
		\leq
		\theta_1-\theta_N
		\leq
		\frac{\pi}{4}.
		\]
		Combining this with \eqref{eq:Dm-variation} proves
		\eqref{eq:lacunary-variation}.
	\end{proof}
	
	\section{Proof of Lemma~\ref{lem:phase-selection}}
	
	\begin{proof}
		Set
		\[
		S(w):=\sum_{j=1}^d\eps_jw_j
		\]
		and define
		\[
		\Gamma(w):=
		\begin{cases}
			\overline{S(w)}/|S(w)|,& S(w)\neq0,\\
			1,& S(w)=0.
		\end{cases}
		\]
		Then $|\Gamma|=1$. The zero set
		\[
		\mathcal N
		:=
		\left\{
		w\in\T^d:
		\sum_{j=1}^d\eps_jw_j=0
		\right\}
		\]
		has Haar measure zero. Indeed, after fixing $w_2,\ldots,w_d$, the defining
		equation admits at most one value of $w_1$, so the assertion follows from
		Fubini's theorem.
		
		For $1\leq j\leq d$, make the change of variables
		\[
		v_\ell=\eps_\ell w_\ell,
		\qquad 1\leq\ell\leq d.
		\]
		Since this transformation preserves Haar measure,
		\[
		\int_{\T^d}\Gamma(w)w_j\,dm_d(w)
		=
		\eps_j
		\int_{\T^d}
		\frac{\overline{\sum_{\ell=1}^dv_\ell}}
		{\left|\sum_{\ell=1}^dv_\ell\right|}
		v_j\,dm_d(v).
		\]
		Permutation invariance of Haar measure shows that the integral on the
		right is independent of $j$, denote its value by $c_d$. Thus
		\[
		\int_{\T^d}\Gamma(w)w_j\,dm_d(w)
		=
		c_d\eps_j.
		\]
		Multiplying by $\eps_j$, summing over $j$ and using the definition of
		$\Gamma$, we obtain
		\[
		dc_d
		=
		\int_{\T^d}|S(w)|\,dm_d(w).
		\]
		After the same change of variables,
		\begin{equation}\label{eq:cd-formula}
			c_d
			=
			\frac1d
			\int_{\T^d}
			\left|\sum_{j=1}^dw_j\right|
			\,dm_d(w).
		\end{equation}
		
		Set
		\[
		F(w):=\sum_{j=1}^dw_j.
		\]
		By orthogonality,
		\[
		\norm{F}_{L^2(\T^d)}^2=d,
		\qquad
		\norm{F}_{L^4(\T^d)}^4=2d^2-d,
		\]
		the latter identity following since only the terms with
		$\{i,j\}=\{k,\ell\}$ survive in the expansion of $\norm{F}_4^4$.
		Interpolation between $L^1$ and $L^4$ gives
		\[
		\norm{F}_{L^2(\T^d)}
		\leq
		\norm{F}_{L^1(\T^d)}^{1/3}\norm{F}_{L^4(\T^d)}^{2/3},
		\]
		and hence
		\[
		\norm{F}_{L^1(\T^d)}
		\geq
		\frac{\norm{F}_{L^2(\T^d)}^3}{\norm{F}_{L^4(\T^d)}^2}
		=
		\frac{d^{3/2}}{\sqrt{2d^2-d}}
		\geq
		\sqrt{\frac d2}.
		\]
		Together with \eqref{eq:cd-formula}, this proves
		\[
		c_d\geq\frac1{\sqrt{2d}}.
		\]
		
		Finally, $\Gamma$ is continuous on $\T^d\setminus\mathcal N$. Since
		$m_d(\mathcal N)=0$ and $\Gamma(w)w_j$ is bounded, the Lebesgue criterion
		for Riemann integrability shows that $w\mapsto\Gamma(w)w_j$ is Riemann
		integrable for every $j$.
	\end{proof}
	
	\section*{Declaration on the use of generative AI}
	
	During the preparation of this work the author used GPT-5.6 Sol by OpenAI in order to assist in the development and refinement of certain mathematical arguments. After using this tool, the author thoroughly reviewed, verified and edited the content as needed and takes full responsibility for the content of the published article.

\end{document}